\documentclass[11pt,leqno]{article}
\usepackage[cp1251]{inputenc}
\usepackage[english]{babel}
\usepackage{amssymb}
\usepackage{amsfonts}
\usepackage{amsthm}
\usepackage{amsmath}

\newtheorem{pro}{Proposition}[section]

\newtheorem{remark}{Remark}[section]
\newtheorem{theorem}{Theorem}[section]
\newtheorem{example}{Example}[section]

			\usepackage[dvips]{color}

\usepackage{graphicx}
\usepackage{pdftricks}
		\usepackage{dsfont}

\begin{document}
\begin{center}
{\large \bf Totally umbilical radical transversal lightlike hypersurfaces of K\"ahler-Norden manifolds of constant totally real
sectional curvatures}
\end{center}

\begin{center}{ Galia  NAKOVA } \end{center}

\begin{center}
{Department of Algebra and Geometry,  \\
Faculty of Mathematics and Informatics,\\
University of Veliko Tarnovo "St. Cyril and St. Methodius",\\
2 T.Tarnovski Str., 5003 Veliko Tarnovo, Bulgaria\\ 
E-mail: gnakova@gmail.com }
\end{center}


				\begin{abstract}
In this paper we study curvature properties of semi-symmetric type of totally umbilical radical transversal lightlike 
		hypersurfaces $(M,g)$ and $(M,\widetilde g)$ of a K\"ahler-Norden manifold $(\overline M,\overline J,\overline g,\overline {     \widetilde g})$ of constant totally real sectional curvatures $\overline \nu$ and $\overline {\widetilde \nu}$
		($g$ and $\widetilde g$ are the induced metrics on $M$ by the Norden metrics $\overline g$ and $\overline {\widetilde g}$,
		respectively). We obtain a condition for $\overline {\widetilde \nu}$ (resp. $\overline \nu$) which is equivalent to each
		of the following conditions: $(M,g)$ $(resp.\, (M,\widetilde g))$ is locally symmetric, semi-symmetric, Ricci semi-symmetric 
		and almost Einstein. We construct an example of a totally umbilical radical transversal lightlike hypersurface, which is 
		locally symmetric, semi-symmetric, Ricci semi-symmetric and almost Einstein.
				\end{abstract}


\section{Introduction}
Almost complex manifolds with Norden metric (or B-metric) were introduced 
by A. P. Norden \cite{Norden} and their geometry  has been investigated by G. Ganchev, K. Gribachev, A. Borisov, V. Mihova 
\cite{GB1, GGM} and other authors. There exists a difference between the geometry of an indefinite almost Hermitian manifold
and the geometry of an almost complex manifold with Norden metric. It arises because in the first case the almost complex 
structure $\overline J$ is an isometry with respect to the semi-Riemannian metric $\overline g$ and in the second case 
$\overline J$ is an anti-isometry with respect to $\overline g$. Due to this property of the couple $(\overline J,\overline g)$
of an almost complex manifold with Norden metric $\overline M$, there exists second Norden metric $\overline {\widetilde g}$ on
$\overline M$ which is defined by $\overline {\widetilde g}(X,Y)=\overline g(\overline JX,Y)$. Both metrics $\overline g$ and
$\overline {\widetilde g}$ are indefinite of a neutral signature. Thus, we can consider two induced metrics $g$ and 
$\widetilde g$ (by $\overline g$ and $\overline {\widetilde g}$, respectively) on a submanifold $M$ of $\overline M$. In 
\cite{GN1} we study submanifolds which are non-degenerate with respect to the one Norden metric and lightlike with respect to
the other one. In \cite{GN} we introduced such class of hypersurfaces, namely, radical transversal lightlike hypersurfaces of almost complex manifolds with Norden metric. As it is well known, in case $M$ is a lightlike submanifold of $\overline M$, a part
of the normal bundle $TM^\bot $ lies in the tangent bundle $TM$. Therefore the geometry of lightlike submanifolds is very different from the Riemannian and the semi-Riemannian geometry. The general theory of lightlike submanifolds has been developed 
by K. Duggal and A. Bejancu in \cite{D-B}. Many new classes of lightlike submanifolds of indefinite Kaehler, Sasakian, 
quaternion Kaehler manifolds are introduced by K. Duggal and B. Sahin in \cite{D-S} and different applications of lightlike geometry in the mathematical physics are given. 
\par An important class of lightlike hypersurfaces are the totally umbilical lightlike hypersurfaces. In \cite{D-B} it is proved
that there exist no totally umbilical lightlike real hypersurfaces of indefinite complex space forms $\overline M(c)$ with
$c\neq 0$. In this paper we consider totally umbilical radical transversal lightlike hypersurfaces of 
K\"ahler-Norden manifolds of constant totally real sectional curvatures $\overline \nu $ and $\overline {\widetilde \nu }$.
Our purpose is to study curvature conditions of semi-symmetric type for the considered hypersurfaces in case 
$(\overline \nu ,\overline {\widetilde \nu })\neq (0,0)$. Lightlike hypersurfaces which are semi-symmetric, Ricci semi-symmetric,
parallel or semi-parallel in a semi-Euclidean  space are investigated in \cite{S}. The curvatures of lightlike hypersurfaces
$M$ of an indefinite Kenmotsu manifold $\overline M$ under the conditions that $M$ is locally symmetric or $M$ is semi-symmetric
are studied in \cite{DHJ}.
\par 
The main results in this paper are given in the following theorem:

\begin{theorem}\label{thm 1.1}
Let $(M,g)$ $(resp.\, (M,\widetilde g))$ be a totally umbilical radical transversal lightlike hypersurface of a K\"ahler-Norden manifold $(\overline M,\overline J,\overline g,\overline {\widetilde g})$ $({\rm dim} \overline M=2n\geq 4)$ of constant 
totally real sectional curvatures $\overline \nu$ and $\overline {\widetilde \nu}$ with respect to $\overline g$ such that 
$\overline {\widetilde \nu}\neq 0$ $(resp.\, \overline \nu\neq 0)$. Then the following assertions are equivalent
\par $(i)$ $(M,g)$ $(resp.\, (M,\widetilde g))$ is semi-symmetric.
\par $(ii)$ $(M,g)$ $(resp.\, (M,\widetilde g))$ is Ricci semi-symmetric.
\par $(iii)$ $\displaystyle{\overline {\widetilde \nu}=\frac{\rho^2}{b}}$ 
$\left(resp.\, \, \displaystyle{\overline \nu=\frac{\widetilde {\rho}^2}{\widetilde b}}\right)$, where $\rho \, (resp. \,
\widetilde \rho)$ and $b \, (resp. \, \widetilde b)$ are the functions from 
\eqref{2.14} (resp. \eqref{3.20}) and \eqref{2.17} (resp. \eqref{3.19}).
\par
Moreover, if $\overline {\widetilde \nu}$ \, (resp. \, $\overline \nu $) is a constant, the assertions 
$(i), (ii), (iii)$ and the assertions 
\par $(iv)$ $(M,g)$ $(resp.\, (M,\widetilde g))$ is locally symmetric.
\par $(v)$ $(M,g)$ $(resp.\, (M,\widetilde g))$ is almost Einstein.\\
are equivalent.
\end{theorem}
We note it is known that locally symmetric manifolds are semi-symmetric but the converse is not true. Moreover, the conditions
the manifold to be semi-symmetric and Ricci semi-symmetric are not equivalent in general.
An application of Theorem \ref{thm 1.1} is Example \ref{exa 5.1}, where we construct a totally umbilical radical transversal lightlike hypersurface, which is locally symmetric, semi-symmetric, Ricci semi-symmetric and almost Einstein.

\section{Preliminaries} 
\subsection{Almost complex manifolds with Norden metric}
A $2n$-dimensional semi-Riemannian manifold $\overline M$ is said to be {\it an almost complex manifold with Norden metric}
(or {\it an almost complex manifold with B-metric}) \cite{GB1} if it is equipped with an almost complex structure $\overline J$
and a semi-Riemannian metric $\overline g$ such that
$$
\overline J^2X=-X, \quad \overline g(\overline JX,\overline JY)=-\overline g(X,Y), \quad X, Y \in \Gamma (T\overline M).
$$
The tensor field $\overline {\widetilde g}$ on $\overline M$ defined by $\overline {\widetilde g}(X,Y)=
\overline g(\overline JX,Y)$ is a Norden metric on $\overline M$, which is said to be {\it an associated metric} of
$\overline M$. Both metrics $\overline g$ and $\overline {\widetilde g}$ are necessarily indefinite of signature $(n,n)$.
The Levi-Civita connection of $\overline g$ (resp. $\overline {\widetilde g}$) is denoted by $\overline \nabla $
(resp. $\overline {\widetilde \nabla }$). The tensor fields $F$ and $\Phi $ are defined by $F(X,Y,Z)=
\overline g((\overline \nabla _X\overline J)Y,Z)$ and $\Phi (X,Y)=\overline {\widetilde \nabla }_XY-\overline \nabla _XY$.
A classification of the almost complex manifolds with Norden metric with respect to the tensor $F$ is given in
\cite{GB1} and eight classes $W_i$ are obtained. These classes are characterized by conditions for the tensor $\Phi $
in \cite{GGM}. The class $W_0$ of the K\"ahler manifolds with Norden metric is determined by the condition $\overline J$
to be parallel with respect to the Levi-Civita connection $\overline \nabla$ of $\overline g$. The class $W_0$ is 
characterized by the following two equivalent conditions $F(X,Y,Z)=0$ and $\Phi (X,Y)=0$. We will call a manifold 
$(\overline M,\overline J,\overline g,\overline {\widetilde g})$ belonging to $W_0$ a K\"ahler-Norden manifold.

\subsection{Lightlike hypersurfaces of semi-Riemannian manifolds}
A hypersurface $M$ of an $(m+2)$-dimensional $(m>1)$ semi-Riemannian manifold $(\overline M,\overline g)$ is called 
{\it a lightlike hypersurface } \cite{D-B, D-S} if at any $p\in M$ the tangent space $T_pM$ and the normal space 
$T_pM^\bot $ have a non-empty intersection, which is denoted by ${\rm Rad} T_pM$. As ${\rm dim} \left(T_pM^\bot \right)=1$ it
follows that ${\rm dim} ({\rm Rad} T_pM)=1$ and ${\rm Rad} T_pM=T_pM^\bot$. The mapping ${\rm Rad} TM : p\in M\longrightarrow
{\rm Rad} T_pM$ defines a smooth distribution on $M$ of rank 1 which is called {\it a radical distribution} on $M$. Thus
the induced metric $g$ by $\overline g$ on a lightlike hypersurface $M$ has a constant rank $m$. There also exists a
non-degenerate complementary vector bundle $S(TM)$ of the normal bundle $TM^\bot $ in the tangent bundle $TM$, called 
{\it a screen distribution} on $M$. We have the following decomposition of $TM$
\begin{equation}\label{2.2}
TM=S(TM)\bot TM^\bot ,
\end{equation} 
 where $\bot $ denotes an orthogonal direct sum. Denote by ${\cal F}(M)$ the algebra of smooth functions on $M$ and
by $\Gamma (E)$ the ${\cal F}(M)$-module of smooth sections of a vector bundle $E$ over $M$. It is well known 
\cite[Theorem 1.1, p. 79]{D-B} that there exists a unique {\it transversal vector bundle} ${\rm tr}(TM)$ of rank 1 over
$M$, such that for any non-zero section $\xi $ of $TM^\bot $ on a coordinate neighbourhood $U\subset M$, there exists a
unique section $N$ of ${\rm tr}(TM)$ on $U$ satisfying
\begin{equation}\label{2.3}
\overline g(N,\xi )=1, \qquad \overline g(N,N)=\overline g(N,W)=0, \, \, \forall \, W \in \Gamma (S(TM)).
\end{equation}
Hence for any screen distribution $S(TM)$ we have a unique ${\rm tr}(TM)$, which is a lightlike complementary vector bundle
(but not orthogonal) to $TM$ in $T\overline M$, such that 
\begin{equation}\label{2.4}
T\overline M=TM\oplus {\rm tr}(TM)=S(TM)\bot \left(TM^\bot \oplus {\rm tr}(TM)\right),
\end{equation}
where $\oplus$ denotes a non-orthogonal direct sum.
\par
Let $\overline \nabla $ be the Levi-Civita connection of $\overline g$ on $\overline M$ and $P$ be the projection morphism
of $\Gamma (TM)$ on $\Gamma (S(TM))$ with respect to the decomposition \eqref{2.2}. The local Gauss and Weingarten formulas of 
$M$ and $S(TM)$ are given by
\begin{equation}\label{2.5}
\overline \nabla _XY=\nabla _XY+B(X,Y)N,
\end{equation}
\begin{equation}\label{2.6}
\overline \nabla _XN=-A_NX+\tau (X)N;  
\end{equation}
and
\begin{equation}\label{2.7}
\nabla _XPY=\nabla ^*_XPY+C(X,PY)\xi ,
\end{equation}
\begin{equation}\label{2.8}
\nabla _X\xi =-A^*_\xi X-\tau (X)\xi ,
\end{equation}
for any $X, Y \in \Gamma (TM)$, respectively. The induced connections $\nabla $ and $\nabla ^*$ on $TM$ and $S(TM)$, respectively,
are linear connections. $A_N$ and $A^*_\xi $ are the shape operators on $TM$ and $S(TM)$, respectively, $\tau $ is a 1-form
on $TM$. Both local second fundamental forms $B$ and $C$ are related to their shape operators by 
\begin{equation}\label{2.9}
B(X,Y)=g(A^*_\xi X,Y), \quad \overline g(A^*_\xi ,N)=0;
\end{equation}
\begin{equation}\label{2.10}
C(X,PY)=g(A_NX,PY), \quad \overline g(A_NX ,N)=0.
\end{equation}
Since $\overline \nabla $ is torsion-free, $\nabla $ is also torsion-free and $B$ is symmetric on $TM$. From \eqref{2.5}
we have $B(X,Y)=\overline g(\overline \nabla_XY,\xi )$, $\forall X, Y\in \Gamma (TM)$ which implies
\begin{equation}\label{2.11}
B(X,\xi )=0, \quad \forall X \in \Gamma (TM).
\end{equation}
From \eqref{2.9} and \eqref{2.11} it follows that the shape operator $A^*_\xi $ is $S(TM)$-valued, self-adjoint with respect
to $g$ and $A^*_\xi \xi =0$. In general, the induced connection $\nabla ^*$ on $S(TM)$ is not torsion-free. This fact and   
\eqref{2.10} show that the shape operator $A_N$ is not self-adjoint and it is $S(TM)$-valued. The linear connection $\nabla ^*$
is a metric connection on $S(TM)$ but $\nabla $ is not metric and satisfies
\begin{equation}\label{2.12}
(\nabla _Xg)(Y,Z)=B(X,Y)\eta (Z)+B(X,Z)\eta (Y),
\end{equation}
where $\eta $ is a 1-form given by
\begin{equation}\label{2.13}
\eta (X)=\overline g(X,N), \quad \forall X\in \Gamma (TM).
\end{equation}
A lightlike hypersurface $M$ is said to be {\it totally umbilical} \cite{D-B} if on any coordinate neighborhood $U$ there
exists a smooth function $\rho $ such that 
\begin{equation}\label{2.14}
B(X,Y)=\rho g(X,Y), \quad \forall X, Y\in \Gamma (TM_{|U}).
\end{equation}
According to \eqref{2.9} and \eqref{2.11} the condition \eqref{2.14} is equivalent to
\begin{equation}\label{2.15}
A^*_\xi (PX)=\rho PX, \quad X\in \Gamma (TM_{|U}).
\end{equation}
Denote by $\overline R$ and $R$ the curvature tensors of $\overline \nabla $ and $\nabla $, respectively. By using \eqref{2.5}
and \eqref{2.6} we get the Gauss equation of $M$
\begin{equation}\label{2.16}
\begin{array}{l}
{\overline R}(X,Y,Z)=R(X,Y,Z)+B(X,Z)A_NY-B(Y,Z)A_NX \\
+\left\{(\nabla _XB)(Y,Z)-(\nabla _YB)(X,Z)+\tau (X)B(Y,Z)-\tau (Y)B(X,Z)\right\}N,  
\end{array}
\end{equation}
for any $X,Y,Z\in \Gamma (TM)$. The induced Ricci type tensor $R^{(0,2)}$ of $M$ is defined by
$R^{(0,2)}(X,Y)={\rm trace}\{Z\longrightarrow R(X,Z)Y\}, \forall X, Y \in \Gamma (TM)$. In general, $R^{(0,2)}$ is not 
symmetric. According to \cite[Theorem 3.2, p. 99]{D-B}, a necessary and sufficient condition for the induced Ricci tensor to be symmetric is each 1-form $\tau $ to be closed, i.e. $d\tau =0$ on $M$.
Therefore $R^{(0,2)}$ is denoted by Ric \cite{D-S} only if the 1-form $\tau $ is closed.

\subsection{Radical transversal lightlike hypersurfaces of a K\"ahler-Norden manifold}
In \cite{GN} we introduced the class of radical transversal lightlike hypersurfaces of an almost complex manifold with Norden metric, which does not exist when the ambient manifold is an indefinite almost Hermitian manifold. 
Let $(M,g,S(TM))$ be a lightlike hypersurface of an almost complex manifold with Norden metric 
$(\overline M,\overline J,\overline g)$. We say that $M$ is a {\it radical transversal lightlike hypersurface} of $\overline M$
if $\overline J(TM^\bot )={\rm tr}(TM)$. For the considered hypersurfaces the following assertions are valid\cite{GN}:
1) $M$ is a radical transversal lightlike hypersurface of $\overline M$ if and only if the screen distribution $S(TM)$ is
holomorphic with respect to $\overline J$;
2) A radical transversal lightlike hypersurface $M$ of $\overline M$ has a unique screen distribution up to a semi-orthogonal
transformation and a unique transversal vector bundle.
\par
There exist two Norden metrics $\overline g$ and $\overline {\widetilde g}$ on an almost complex manifold with Norden metric 
$\overline M$. Hence we can consider two induced metrics $g$ and $\widetilde g$ on a hypersurface $M$ of $\overline M$ by
$\overline g$ and $\overline {\widetilde g}$, respectively. Denote by $(M,g)$ a non-degenerate hypersurface of 
$(\overline M,\overline J,\overline g,\overline {\widetilde g})$, whose normal vector field $\overline N$ is a space-like
unit $(\overline g(\overline N,\overline N)=1)$ or a time-like unit $(\overline g(\overline N,\overline N)=-1)$ and 
$\overline N$ is orthogonal to $\overline J\overline N$ with respect to $\overline g$. Then we proved \cite{GN} that 
$(M,g)$ is a non-degenerate hypersurface of $\overline M$ if and only if $(M,\widetilde g)$ is a radical transversal
lightlike hypersurface.
\par
Further, we consider a radical transversal lightlike hypersurface $(M,g,S(TM))$ of a K\"ahler-Norden manifold 
$(\overline M,\overline J,\overline g,\overline {\widetilde g})$. Let $\{\xi ,N\}$ be the pair on $(M,g)$ which satisfies
condition \eqref{2.3}. From the definition of $M$ we have 
\begin{equation}\label{2.17}
\overline J\xi =bN,
\end{equation}
where $b\in {\cal F}(\overline M)$. Taking into account \eqref{2.2} and \eqref{2.13}, for an arbitrary $X\in \Gamma (TM)$
we have the following decomposition $X=PX+\eta (X)\xi $. From the last equality and \eqref{2.17} we obtain
\begin{equation}\label{2.18}
\overline JX=\overline J(PX)+b\eta (X)N .
\end{equation}
Since $S(TM)$ is holomorphic with respect to $\overline J$, it follows that $\overline J(PX)$ belongs to $S(TM)$.
The shape operators $A^*_\xi $, $A_N$  and the corresponding local second fundamental forms $B$, $C$ are related as follows
\begin{equation}\label{2.19}
A^*_\xi X=-b\overline J(A_NX), \quad B(X,Y)=-bC(X,\overline J(PY)), \quad \forall X, Y\in \Gamma (TM). 
\end{equation}
Moreover, the metric linear connection $\nabla ^*$ of the considered hypersurfaces is such that the almost complex structure
$\overline J$ is parallel with respect to $\nabla ^*$, i.e.
\begin{equation}\label{2.20} 
(\nabla ^*_X\overline J)PY=0, \quad \forall X, Y\in \Gamma (TM).
\end{equation}
The 1-form $\tau $ is expressed by the function $b$ in the following way
\begin{equation}\label{2.21}
\tau (X)=-\frac{1}{2b}X(b),
\end{equation}
which means that $\tau $ is closed and consequently the induced Ricci tensor on $(M,g)$ is symmetric.

\section{Totally umbilical radical transversal lightlike hypersurfaces of a K\"ahler-Norden manifold of constant totally
real sectional curvatures}
The characteristic condition $\Phi =0$ of a $2n$-dimensional K\"ahler-Norden manifold $(\overline M,\overline J,\overline g,\overline {\widetilde g})$ implies that the Levi-Civita connections $\overline \nabla $ and $\overline {\widetilde \nabla }$
of $\overline M$ coincide. Hence the curvature tensors of type (1,3) $\overline R$ and $\overline {\widetilde R}$ of
$\overline \nabla $ and $\overline {\widetilde \nabla }$ coincide, too. Further, $\overline X,\overline Y,\overline Z,
\overline W$ (resp. $\overline x,\overline y,\overline z,\overline w$) will stand for arbitrary differentiable vector fields on
$\overline M$ (resp. vectors in $T_p\overline M, p\in \overline M$). The curvature tensors $\overline R$ and 
$\overline {\widetilde R}$ of type (0,4) are given by 
$\overline R(\overline X,\overline Y,\overline Z,\overline W)=\overline g(\overline R(\overline X,\overline Y,\overline Z),
\overline W)$ and
\begin{equation}\label{3.1}
\overline {\widetilde R}(\overline X,\overline Y,\overline Z,\overline W)=\overline {\widetilde g}(\overline {\widetilde R}
(\overline X,\overline Y,\overline Z),\overline W)=\overline R(\overline X,\overline Y,\overline Z,\overline J\overline W).
\end{equation}
From the condition $\overline \nabla \overline J=\overline J\overline \nabla$ it follows $\overline R(\overline X,\overline Y,\overline J\overline Z)=\overline J\overline R(\overline X,\overline Y,\overline Z)$. The last equality and the fact that
$\overline J$ is an anti-isometry with respect to $\overline g$ and $\overline {\widetilde g}$ give
\begin{equation}\label{3.2}
\overline R(\overline X,\overline Y,\overline J\overline Z,\overline J\overline W)=-\overline R(\overline X,\overline Y,
\overline Z,\overline W), \, \,  \overline {\widetilde R}(\overline X,\overline Y,\overline J\overline Z,\overline J\overline W)=
-\overline {\widetilde R}(\overline X,\overline Y,\overline Z,\overline W),
\end{equation}
i.e. $\overline R$ and $\overline {\widetilde R}$ are K\"ahler tensors. The following tensors are essential in the geometry
of the K\"ahler-Norden manifolds
\begin{equation}\label{3.3}
\begin{array}{lll}
\overline \pi _1(\overline X,\overline Y,\overline Z,\overline W)=\overline g(\overline Y,\overline Z)
\overline g(\overline X,\overline W)-\overline g(\overline X,\overline Z)\overline g(\overline Y,\overline W),\\
\overline \pi _2(\overline X,\overline Y,\overline Z,\overline W)=\overline g(\overline Y,\overline J\overline Z)
\overline g(\overline X,\overline J\overline W)-\overline g(\overline X,\overline J\overline Z)
\overline g(\overline Y,\overline J\overline W),\\
\overline \pi _3(\overline X,\overline Y,\overline Z,\overline W)=-\overline g(\overline Y,\overline Z)
\overline g(\overline X,\overline J\overline W)+\overline g(\overline X,\overline Z)
\overline g(\overline Y,\overline J\overline W)\\
\qquad \qquad \qquad \quad -\overline g(\overline X,\overline W)\overline g(\overline Y,\overline J\overline Z)+
\overline g(\overline Y,\overline W)\overline g(\overline X,\overline J\overline Z).
\end{array}
\end{equation}
By $\overline {\widetilde\pi }_i\, (i=1,2,3)$ are denoted the corresponding tensors with respect to $\overline {\widetilde g}$.
They are related with $\overline \pi _i\, (i=1,2,3)$  as follows
\begin{equation}\label{3.4}
\overline {\widetilde\pi }_1=\overline \pi _2, \quad \overline {\widetilde\pi }_2=\overline \pi _1, \quad
\overline {\widetilde\pi }_3=-\overline \pi _3.
\end{equation}
For every non-degenerate section $\alpha =span\{\overline x,\overline y\}$ with respect to $\overline g$ in $T_p\overline M$
the following two sectional curvatures $\overline \nu $ and $\overline {\widetilde \nu }$ are defined in \cite{GB2}
$$
\overline \nu (\alpha ;p)=\frac{\overline R(\overline x,\overline y,\overline y,\overline x)}
{\overline \pi _1(\overline x,\overline y,\overline y,\overline x)} , \quad
\overline {\widetilde \nu }(\alpha ;p)=\frac{\overline R(\overline x,\overline y,\overline y,\overline J\overline x)}
{\overline \pi _1(\overline x,\overline y,\overline y,\overline x)} .
$$
Analogously, if $\alpha =span\{\overline x,\overline y\}$ is a non-degenerate section with respect to $\overline {\widetilde g}$,
we can define two sectional curvatures $\overline \nu ^\prime $ and $\overline {\widetilde \nu }^\prime $ given by
$$
\overline \nu ^\prime (\alpha ;p)=\frac{\overline {\widetilde R}(\overline x,\overline y,\overline y,\overline x)}
{\overline {\widetilde\pi }_1(\overline x,\overline y,\overline y,\overline x)} , \quad
\overline {\widetilde \nu }^\prime (\alpha ;p)=
\frac{\overline {\widetilde R}(\overline x,\overline y,\overline y,\overline J\overline x)}
{\overline {\widetilde\pi }_1(\overline x,\overline y,\overline y,\overline x)} .
$$
A section $\alpha $ is said to be {\it holomorphic} if $\overline J\alpha =\alpha $ and its sectional curvature is called {\it a
holomorphic sectional curvature}. A section $\alpha $ is said to be {\it totally real} with respect to $\overline  g$ 
(resp. $\overline {\widetilde g}$) if $\overline J\alpha $ is orthogonal to $\alpha $ with respect to $\overline  g$ 
(resp. $\overline {\widetilde g}$). We will call the sectional curvatures $\overline \nu $ and $\overline {\widetilde \nu }$
(resp. $\overline \nu ^\prime $ and $\overline {\widetilde \nu }^\prime $) of a non-degenerate totally real section with 
respect to $\overline  g$  (resp. $\overline {\widetilde g}$) {\it totally real sectional curvatures} with respect to
$\overline  g$ (resp. $\overline {\widetilde g}$). It is well known that {\it an indefinite complex space form} is a connected indefinite K\"ahler manifold $\overline M$ of constant holomorphic sectional curvature $c$ and it is denoted by $\overline M(c)$. When $\overline M$ is a K\"ahler-Norden manifold the property \eqref{3.2} gives that all of the holomorphic sectional curvatures   are zero. Therefore the sectional curvatures of the totally real sections are important in our considerations. 
\begin{theorem}\label{thm 3.1}
\cite{GB2} Let $(\overline M,\overline J,\overline g,\overline {\widetilde g})$ be a $2n$-dimensional ($2n\geq 4$)
K\"ahler-Norden manifold. $\overline M$ is of constant totally real sectional curvatures $\overline \nu $ and 
$\overline {\widetilde \nu }$ with respect to $\overline g$, i.e. $\overline \nu (\alpha ;p)=\overline \nu (p)$,
$\overline {\widetilde \nu }(\alpha ;p)=\overline {\widetilde \nu }(p)$ \, ($p\in \overline M$), if and only if 
\begin{equation}\label{3.5}
\overline R=\overline \nu[\overline \pi _1-\overline \pi _2]+\overline {\widetilde \nu }\overline \pi _3 .
\end{equation}
Both functions $\overline \nu $ and $\overline {\widetilde \nu }$ are constants if $\overline M$ is connected and $2n\geq 6$.
\end{theorem}
\begin{remark}\label{rem 3.1}
Suppose $\overline M$ is of constant totally real sectional curvatures $\overline \nu $ and $\overline {\widetilde \nu }$ with respect to $\overline g$. Then from \eqref{3.5} by using \eqref{3.1}, \eqref{3.3} and \eqref{3.4} we obtain \\
$\overline {\widetilde R}=-\overline {\widetilde \nu }[\overline {\widetilde \pi} _1-\overline {\widetilde \pi } _2]
+\overline \nu \overline {\widetilde \pi } _3$.
Then from Theorem \ref{thm 3.1} it follows that $\overline M$ is of
constant totally real sectional curvatures $\overline \nu ^\prime =-\overline {\widetilde \nu }$ and 
$\overline {\widetilde \nu }^\prime =\overline \nu$ with respect to $\overline {\widetilde g}$.
\end{remark}
\begin{theorem}\label{thm 3.2}
Let $(M,g)$ be a totally umbilical radical transversal lightlike hypersurface of a K\"ahler-Norden manifold 
$(\overline M,\overline J,\overline g,\overline {\widetilde g})$ (${\rm dim} \overline M=2n\geq 4$) of constant totally
real sectional curvatures $\overline \nu $ and $\overline {\widetilde \nu }$ with respect to $\overline g$.
Then $\overline \nu =0$ and $\rho $ from \eqref{2.14} satisfies the partial
differential equations 
\begin{equation}\label{3.7}
b\overline {\widetilde \nu }-\rho ^2+\xi (\rho )+\rho \tau (\xi )=0,
\end{equation}
\begin{equation}\label{3.8}
PX(\rho )+\rho \tau (PX)=0, \quad \forall X\in \Gamma (TM),
\end{equation}
where $b$ is the function from \eqref{2.17}.
Moreover, the curvature tensor $R$ and the Ricci tensor Ric of $(M,g)$ are given by 
\begin{equation}\label{3.9}
\begin{array}{ll}
R(X,Y,Z)=\displaystyle{\left(\overline {\widetilde \nu }-\frac{\rho ^2}{b}\right)\left[g(X,Z)\overline J(PY)-g(Y,Z)\overline J(PX)\right]}\\
\qquad \qquad \quad +\overline {\widetilde \nu }\left[\overline g(X,\overline JZ)Y-\overline g(Y,\overline JZ)X\right],
\end{array}
\end{equation}
\begin{equation}\label{3.10}
{\rm Ric}(X,Y)=-2(n-1)\overline {\widetilde \nu }{\widetilde g}(X,Y)+
\displaystyle{\left(\overline {\widetilde \nu }-\frac{\rho ^2}{b}\right){\widetilde g}(PX,PY)},
\end{equation}
for any $X, Y, Z\in \Gamma (TM)$.
\end{theorem}
\begin{proof}
By using \eqref{2.16}, \eqref{3.5}, \eqref{2.17} and \eqref{2.14} we obtain
$$
\begin{array}{lll}
b\overline \nu \left[-\overline g(Y,\overline JZ)\eta (X)+\overline g(X,\overline JZ)\eta (Y)\right]+
b\overline {\widetilde \nu }\left[-g(Y,Z)\eta (X)+g(X,Z)\eta (Y)\right]\\
=X(\rho )g(Y,Z)-Y(\rho )g(X,Z)\\
+\rho \left[(\nabla _Xg)(Y,Z)-(\nabla _Yg)(X,Z)+\tau (X)g(Y,Z)-\tau (Y)g(X,Z)\right],\, \forall X, Y, Z\in \Gamma (TM).
\end{array}
$$
Taking into account \eqref{2.12} and \eqref{2.14}, the above equality becomes
\begin{equation}\label{3.12}
\begin{array}{ll}
b\overline \nu \left[-\overline g(Y,\overline JZ)\eta (X)+\overline g(X,\overline JZ)\eta (Y)\right]+
b\overline {\widetilde \nu }\left[-g(Y,Z)\eta (X)+g(X,Z)\eta (Y)\right]\\
=X(\rho )g(Y,Z)-Y(\rho )g(X,Z)\\
+\rho \left[\rho g(X,Z)\eta (Y)-\rho g(Y,Z)\eta (X)+\tau (X)g(Y,Z)-\tau (Y)g(X,Z)\right].
\end{array}
\end{equation}
Replacing $X, Y$ and $Z$ in \eqref{3.12} by $PX, \xi $ and $PZ$, respectively, we get
$$
b\overline \nu g(PX,\overline J(PZ))+b\overline {\widetilde \nu }g(PX,PZ)=\left[\rho ^2-\xi (\rho )-\rho \tau (\xi )\right]
g(PX,PZ).
$$
Because $S(TM)$ is non-degenerate we have
$$
\left[b\overline {\widetilde \nu }-\rho ^2+\xi (\rho )+\rho \tau (\xi )\right]PZ+b\overline \nu \overline J(PZ)=0. 
$$
Since $PZ$ and $\overline J(PZ)$ are linearly independent and $b\neq 0$, we obtain \eqref{3.7} and $\overline \nu =0$.
Now if we take $X=PX, Y=PY, Z=PZ$ in \eqref{3.12} and by using $S(TM)$ is non-degenerate we have
\begin{equation}\label{3.13}
\left[PX(\rho )+\rho \tau (PX)\right]PY=\left[PY(\rho )+\rho \tau (PY)\right]PX.
\end{equation}
Suppose there exists a vector field $X_0\in \Gamma (TM_{|U})$ such that $PX_0(\rho )+\rho \tau (PX_0)\neq 0$ at a point
$p\in M$. Then from \eqref{3.13} it follows that all vectors from $S(TM)$ are collinear with $(PX_0)_p$. This is a 
contradiction because ${\rm dim}S(TM)=2n-2\geq 2$. Hence \eqref{3.8} is true at any point $p\in M$. Next,
by using \eqref{3.5} and taking into account that $\overline \nu =0$ and \eqref{2.18}, we obtain 
\begin{equation}\label{3.15}
\begin{array}{ll}
\overline R(X,Y,Z)=\overline {\widetilde \nu }\left[-g(Y,Z)\overline J(PX)+g(X,Z)\overline J(PY)\right. \\
\qquad \qquad \quad \left.-\overline g(Y,\overline JZ)X+\overline g(X,\overline JZ)Y\right].
\end{array}
\end{equation}
The equalities \eqref{2.15} and \eqref{2.19} imply
\begin{equation}\label{3.16}
A_NPX=\frac{\rho }{b}\overline J(PX), \quad \forall X\in \Gamma (TM).
\end{equation}
So, \eqref{3.9} follows from \eqref{2.16}, \eqref{3.15}, \eqref{3.16} and \eqref{2.14}. Finally, we obtain \eqref{3.10} by
using the expression of the Ricci tensor $R^{(0,2)}$ of a lightlike hypersurface \cite{D-S}, given by
\begin{equation}\label{3.17}
R^{(0,2)}(X,Y)=\overline {\rm Ric}(X,Y)+B(X,Y){\rm tr}A_N-g(A_NX,A^*_\xi Y)-\overline g(R(\xi ,Y,X),N),
\end{equation}
where $\overline {\rm Ric}$ is the Ricci tensor of $\overline M$. From \eqref{3.16} it follows that ${\rm tr}A_N=0$. 
By using \eqref{3.5} and $\overline \nu =0$ we find
\begin{equation}\label{3.18}
\overline {\rm Ric}(X,Y)=-2(n-1)\overline {\widetilde \nu }\overline g(X,\overline JY), \quad \forall X, Y\in \Gamma (TM).
\end{equation}
Substituting \eqref{3.18} in \eqref{3.17} and taking into account \eqref{2.15}, \eqref{3.9} and \eqref{3.16} we get \eqref{3.10}.
\end{proof}
Now, let $(M,{\widetilde g})$ be a radical transversal lightlike hypersurface of a K\"ahler-Norden manifold 
$(\overline M,\overline J,\overline g,\overline {\widetilde g})$. Denote by $\widetilde {TM}, \widetilde {TM}^{\widetilde {\bot }},
S(\widetilde {TM})$ and ${\rm tr}(\widetilde {TM})$ the tangent bundle, the normal bundle, the screen distribution and the
transversal vector bundle of $(M,{\widetilde g})$, respectively. Let $\{\xi ,N\}$ be the pair of vector fields 
$\xi \in \Gamma (\widetilde {TM}^{\widetilde {\bot }})$, $N\in \Gamma ({\rm tr}(\widetilde {TM}))$ satisfying 
\begin{equation}\label{3.19}
\overline J\xi =\widetilde bN, \, \, 
\overline {\widetilde g}(N,\xi )=1,\, \, \overline {\widetilde g}(N,N)=\overline {\widetilde g}(N,W)=0, \, \, \forall W\in 
\Gamma (S(\widetilde {TM})).
\end{equation}  
The local second fundamental form $\widetilde B$ and the 1-form $\widetilde \tau $ are given by 
\begin{equation}\label{3.20}
\widetilde B(X,Y)=\widetilde  \rho \widetilde g(X,Y), \quad  \widetilde \tau (X)=-\frac{1}{2\widetilde b}X(\widetilde b),
\quad \forall X, Y\in \Gamma (\widetilde {TM}).
\end{equation}
Taking into account Remark \ref{rem 3.1}, the following theorem follows in a similar way as Theorem \ref{thm 3.2} 
\begin{theorem}\label{thm 3.3}
Let $(M,\widetilde g)$ be a totally umbilical radical transversal lightlike hypersurface of a K\"ahler-Norden manifold 
$(\overline M,\overline J,\overline g,\overline {\widetilde g})$ (${\rm dim} \overline M=2n\geq 4$) of constant totally
real sectional curvatures $\overline \nu $ and $\overline {\widetilde \nu }$ with respect to $\overline g$. 
Then $\overline {\widetilde \nu} =0$ and the function $\widetilde \rho $ 
from \eqref{3.20} satisfies the partial differential equations 
\begin{equation}\label{3.21}
\widetilde b\overline \nu -\widetilde \rho ^2+\xi (\widetilde \rho )+\widetilde \rho \widetilde {\tau }(\xi )=0,
\end{equation}
\begin{equation}\label{3.22}
PX(\widetilde \rho )+\widetilde \rho \widetilde {\tau }(PX)=0, \quad \forall X\in \Gamma (\widetilde {TM}),
\end{equation}
where $\widetilde b$ is the function from \eqref{3.19}.
Moreover, the curvature tensor $\widetilde R$ and the Ricci tensor $\widetilde {\rm Ric}$ of $(M,\widetilde g)$ are given by 
\begin{equation}\label{3.23}
\begin{array}{ll}
\widetilde R(X,Y,Z)=\displaystyle{\left(\overline \nu -\frac{\widetilde \rho ^2}{\widetilde b}\right)\left[\widetilde g(X,Z)
\overline J(PY)-\widetilde g(Y,Z)\overline J(PX)\right]}\\
\qquad \qquad \quad +\overline \nu \left[\overline {\widetilde g}(X,\overline JZ)Y-
\overline {\widetilde g}(Y,\overline JZ)X\right],
\end{array}
\end{equation}
\begin{equation}\label{3.24}
\widetilde {\rm Ric}(X,Y)=2(n-1)\overline \nu g(X,Y)-
\displaystyle{\left(\overline \nu -\frac{\widetilde \rho ^2}{\widetilde b}\right)g(PX,PY)},
\end{equation}
for any $X, Y, Z\in \Gamma (\widetilde {TM})$.
\end{theorem}

\section{Proof of Theorem \ref{thm 1.1}}
\begin{proof}
Let $(M,g)$ be a totally umbilical radical transversal lightlike hypersurface of $\overline M$. From Theorem \ref{thm 3.2} it
follows that $\overline \nu =0$. Denote $\displaystyle{a=\overline {\widetilde \nu}-\frac{\rho^2}{b}}$.
\par
$(i)\Longleftrightarrow (iii)$ A lightlike hypersurface $M$ is said to be {\it semi-symmetric} \cite{D-S} if \\
$({\cal R}(X,Y)\cdot R)(U,V,W)=0$, where
\begin{equation}\label{4.1}
\begin{array}{ll}
({\cal R}(X,Y)\cdot R)(U,V,W)=R(X,Y,R(U,V,W))-R(U,V,R(X,Y,W)) \\
-R(R(X,Y,U),V,W)-R(U,R(X,Y,V),W), \quad \forall X, Y, U, V, W\in \Gamma (TM).
\end{array}
\end{equation}
By using \eqref{3.9} and \eqref{4.1} we obtain
$$
\begin{array}{lllll}
({\cal R}(X,Y)\cdot R)(U,V,W)=a\left\{g(V,W)\left[ag(Y,\overline J(PU))-\overline {\widetilde \nu}
\overline g(Y,\overline JU)\right]\right. \\
\qquad \qquad \qquad \qquad \qquad\left.-g(U,W)\left[ag(Y,\overline J(PV))-\overline {\widetilde \nu}
\overline g(Y,\overline JV)\right]\right\}\overline J(PX)\\
\qquad \qquad \qquad \qquad \qquad -a\left\{g(V,W)\left[ag(X,\overline J(PU))-\overline {\widetilde \nu}
\overline g(X,\overline JU)\right]\right. \\
\qquad \qquad \qquad \qquad \qquad\left.-g(U,W)\left[ag(X,\overline J(PV))-\overline {\widetilde \nu}
\overline g(X,\overline JV)\right]\right\}\overline J(PY)\\
\qquad \qquad \qquad \qquad \qquad +a\left[g(U,W)g(V,Y)-g(V,W)g(U,Y)\right]\left[\overline {\widetilde \nu}X-aPX\right]\\
\qquad \qquad \qquad \qquad \qquad -a\left[g(U,W)g(V,X)-g(V,W)g(U,X)\right]\left[\overline {\widetilde \nu}Y-aPY\right]\\

+a\left\{g(Y,U)\left[ag(W,\overline J(PX))-\overline {\widetilde \nu}
\overline g(W,\overline JX)\right]
-g(X,U)\left[ag(W,\overline J(PY))-\overline {\widetilde \nu}
\overline g(W,\overline JY)\right]\right.\\
\qquad \qquad \qquad \qquad \qquad \left.+g(Y,W)\left[ag(U,\overline J(PX))-\overline {\widetilde \nu}
\overline g(U,\overline JX)\right]\right. \\
\qquad \qquad \qquad \qquad \qquad \left.-g(X,W)\left[ag(U,\overline J(PY))-\overline {\widetilde \nu}
\overline g(U,\overline JY)\right]\right\}\overline J(PV)\\
-a\left\{g(Y,V)\left[ag(W,\overline J(PX))-\overline {\widetilde \nu}
\overline g(W,\overline JX)\right]
-g(X,V)\left[ag(W,\overline J(PY))-\overline {\widetilde \nu}
\overline g(W,\overline JY)\right]\right.\\
\qquad \qquad \qquad \qquad \qquad \left.+g(Y,W)\left[ag(V,\overline J(PX))-\overline {\widetilde \nu}
\overline g(V,\overline JX)\right]\right. \\
\qquad \qquad \qquad \qquad \qquad \left.-g(X,W)\left[ag(V,\overline J(PY))-\overline {\widetilde \nu}
\overline g(V,\overline JY)\right]\right\}\overline J(PU).
\end{array}
$$
Assume that $(M,g)$ is semi-symmetric. Then $({\cal R}(X,Y)\cdot R)(U,V,W)=0$ for any $X, Y, U, V, W\in \Gamma (TM)$.
If we take $X=PX, Y=\xi , U=\xi , V=PV, W=PW$ and by using \eqref{2.17} we get $ab\overline {\widetilde \nu}\left[-g(PV,PW)
\overline J(PX)+g(PX,PW)\overline J(PV)\right]=0$ for any $X, V, W\in \Gamma (TM)$. 
Suppose $a\neq 0$. Since $b \neq 0$ and $\overline {\widetilde \nu} \neq 0$  it follows that \\ 
$g(PV,PW)\overline J(PX)=g(PX,PW)\overline J(PV)$ for any
$X, V, W\in \Gamma (TM)$. Applying $\overline J$ to the last equality and replacing $W$ by $V$ we have 
\begin{equation}\label{4.11}
g(PV,PV)PX=g(PX,PV)PV.
\end{equation}
Suppose there exists a vector field $V_0\in \Gamma (TM)$ such that $g({PV}_0,{PV}_0)\neq 0$ at some point $p\in M$.
Then from \eqref{4.11} it follows that all vectors from $S(TM)$ are collinear with $PV_0$. This is a 
contradiction because ${\rm dim}S(TM)=2n-2\geq 2$. Hence $a=0$, i.e. $\displaystyle{\overline {\widetilde \nu}=\frac{\rho^2}{b}}$
at any point $p\in M$. 
Conversely, substituting $a=0$ in $({\cal R}(X,Y)\cdot R)(U,V,W)$, we obtain $(i)$.
\par
$(ii)\Longleftrightarrow (iii)$ A lightlike hypersurface $M$ is said to be {\it Ricci semi-symmetric} \cite{D-S} if 
$({\cal R}(X,Y)\cdot {\rm Ric})(X_1,X_2)=0$, where
\begin{equation}\label{4.2}
({\cal R}(X,Y)\cdot {\rm Ric})(X_1,X_2)=-{\rm Ric}(R(X,Y,X_1),X_2)-{\rm Ric}(X_1,R(X,Y,X_2)),  
\end{equation}
for any $X, Y, X_1, X_2\in \Gamma (TM)$. By using \eqref{3.9}, \eqref{3.10}, \eqref{4.2} and \eqref{2.18} we have
\begin{equation}\label{4.14}
\begin{array}{lrr}
({\cal R}(X,Y)\cdot {\rm Ric})(X_1,X_2)\\
=-ab\overline {\widetilde \nu}\eta (X_1)\left[\eta (X)g(Y,\overline J(PX_2))-
\eta (Y)g(X,\overline J(PX_2))\right]\\ 
-ab\overline {\widetilde \nu}\eta (X_2)\left[\eta (X)g(Y,\overline J(PX_1))-
\eta (Y)g(X,\overline J(PX_1))\right].
\end{array}
\end{equation}
Assume that $(M,g)$ is Ricci semi-symmetric. Then $({\cal R}(X,Y)\cdot {\rm Ric})(X_1,X_2)=0$ for any 
$X, Y, X_1, X_2\in \Gamma (TM)$. Taking $X=\xi , Y=PY, X_1=PX_1, X_2=\xi $ we obtain
$ab\overline {\widetilde \nu}g(PY,\overline J(PX_1))=0$. From the last equality, taking into account that 
$b\neq 0$, $\overline {\widetilde \nu}\neq 0$ and that $S(TM)$ is non-degenerate , it follows that $a=0$.
The implication $(iii)\Longrightarrow (ii)$ follows from \eqref{4.14}.
\par
Now, let $\overline {\widetilde \nu}$ be a constant. We note that according to Theorem \ref{thm 3.1} 
$\overline \nu $ and $\overline {\widetilde \nu}$ are always constants if $\overline M$ is connected and 
${\rm dim} \overline M=2n\geq 6$.
\par
$(iv)\Longleftrightarrow (iii)$ A lightlike hypersurface $M$ is said to be {\it locally symmetric} \cite{D-S} if \\
$(\nabla _UR)(X,Y,Z)=0$, where
\begin{equation}\label{4.111}
\begin{array}{ll}
(\nabla _UR)(X,Y,Z)=\nabla _UR(X,Y,Z)-R(\nabla _UX,Y,Z)-R(X,\nabla _UY,Z) \\
\qquad \qquad \qquad \quad -R(X,Y,\nabla _UZ), \quad \forall X, Y, Z, U\in \Gamma (TM).
\end{array}
\end{equation}
By using \eqref{3.9}, \eqref{4.111} we obtain
\begin{equation}\label{4.4}
\begin{array}{lll}
(\nabla _UR)(X,Y,Z)=-\left[a(\nabla _Ug)(Y,Z)+g(Y,Z)U(a)\right]\overline J(PX)\\
+\left[a(\nabla _Ug)(X,Z)+g(X,Z)U(a)\right]\overline J(PY)-\overline {\widetilde \nu}
\left[U(\overline g(Y,\overline JZ))-\overline g(\nabla _UY,\overline JZ)\right. \\
\left.-\overline g(Y,\overline J(\nabla _UZ))\right]X+
\overline {\widetilde \nu}\left[U(\overline g(X,\overline JZ))-\overline g(\nabla _UX,\overline JZ)
-\overline g(X,\overline J(\nabla _UZ))\right]Y\\
\qquad \qquad \qquad \qquad -ag(Y,Z)\left[\nabla _U\overline J(PX)-\overline J(P(\nabla _UX))\right]\\
\qquad \qquad \qquad \qquad +ag(X,Z)\left[\nabla _U\overline J(PY)-\overline J(P(\nabla _UY))\right].
\end{array}
\end{equation}
Consider the following expressions from \eqref{4.4} 
$$
\begin{array}{ll}
(A)=a(\nabla _Ug)(X,Z)+g(X,Z)U(a), \\
(B)=U(\overline g(X,\overline JZ))-\overline g(\nabla _UX,\overline JZ)-\overline g(X,\overline J(\nabla _UZ)),\\
(C)=\nabla _U\overline J(PX)-\overline J(P(\nabla _UX)).
\end{array}
$$
For the expression $(A)$ by using \eqref{2.12} and \eqref{2.14} we get
\begin{equation}\label{4.5}
(A)=a\rho \left[g(U,X)\eta (Z)+g(U,Z)\eta (X)\right]+g(X,Z)U(a).
\end{equation}
Taking into account \eqref{2.5} and the fact that $\overline \nabla $ is a metric connection, the expression $(B)$ becomes 
$$
(B)=\overline g((\overline \nabla _U\overline J)Z,X)-\overline g((\overline \nabla _U\overline J)X,Z)+
B(U,X)\overline g(Z,\overline JN)+B(U,Z)\overline g(X,\overline JN).
$$
Then \eqref{2.17} and $\overline \nabla \overline J=0$ imply that $(B)$ vanishes for any $X, Z, U\in \Gamma (TM)$.
By using $\overline J(PX)\in S(TM)$, \eqref{2.7}, \eqref{2.14} and \eqref{2.19}, for the first term of $(C)$ we have
\begin{equation}\label{4.6}
\nabla _U\overline J(PX)=\nabla ^*_U\overline J(PX)-\frac{\rho }{b}g(U,X)\xi .
\end{equation}
According to \eqref{2.7}, \eqref{2.8} and \eqref{2.15}, the second term of $(C)$ becomes
\begin{equation}\label{4.7}
\overline J(P(\nabla _UX))=\overline J(\nabla ^*_UPX)-\rho \eta (X)\overline J(PU).
\end{equation}
From \eqref{4.6}, \eqref{4.7} and \eqref{2.20} we obtain
\begin{equation}\label{4.8}
(C)=-\frac{\rho }{b}g(U,X)\xi +\rho \eta (X)\overline J(PU).
\end{equation}
Now, \eqref{4.5}, $(B)=0$ and \eqref{4.8} imply
\begin{equation}\label{4.9}
\begin{array}{llll}
(\nabla _UR)(X,Y,Z)=-\left\{a\rho \left[g(U,Y)\eta (Z)+g(U,Z)\eta (Y)\right]\right. \\
\left.+g(Y,Z)U(a)\right\}\overline J(PX)+
\left\{a\rho \left[g(U,X)\eta (Z)+g(U,Z)\eta (X)\right]\right. \\
\left.+g(X,Z)U(a)\right\}\overline J(PY)-ag(Y,Z)\left[-\displaystyle{\frac{\rho }{b}}g(U,X)\xi +
\rho \eta (X)\overline J(PU)\right] \\
+ag(X,Z)\left[-\displaystyle{\frac{\rho }{b}}g(U,Y)\xi +\rho \eta (Y)\overline J(PU)\right].
\end{array}
\end{equation}
Assume that $(M,g)$ is locally symmetric. Then for any $X, Y, Z, U\in \Gamma (TM)$ we have
$(\nabla _UR)(X,Y,Z)=0$. Taking $Z=\xi $ in \eqref{4.9} we obtain
$$
a\rho \left[-g(U,Y)\overline J(PX)+g(U,X)\overline J(PY)\right]=0, \quad \forall X, Y, U\in \Gamma (TM).
$$
Taking into account that $\rho \neq 0$ we conclude that $a=0$ in the same way as in the implication $(i)\Longrightarrow (iii)$.
The implication $(iii)\Longrightarrow (iv)$ follows from \eqref{4.9}.

\par
We say that a lightlike hypersurface $M$ of an almost complex manifold
with Norden metric $(\overline M,\overline J,\overline g,\overline {\widetilde g})$ is {\it almost Einstein} if 
${\rm Ric}=kg+c\widetilde g$, where $k$ and $c$ are constants and $g$, $\widetilde g$ are the induced metrics on $M$
by $\overline g$, $\overline {\widetilde g}$, respectively.
\par
$(v)\Longleftrightarrow (iii)$ Let $(M,g)$ be almost Einstein. Then by using \eqref{3.10} we have
$$
kg(X,Y)+c\widetilde g(X,Y)=-2(n-1)\overline {\widetilde \nu }\widetilde g(X,Y)+a\widetilde g(PX,PY), \quad X, Y\in
\Gamma (TM).
$$
If we replace $X$ and $Y$ from the above equality by $PX$ and $PY$, respectively, we obtain
\begin{equation}\label{4.17}
kg(PX,PY)+\left[c+2(n-1)\overline {\widetilde \nu }-a\right]g(PX,\overline J(PY))=0.
\end{equation} 
Since $S(TM)$ is non-degenerate, from \eqref{4.17} it follows that
\begin{equation}\label{4.18}
\displaystyle{kPX+\left[c+(2n-3)\overline {\widetilde \nu }+\frac{\rho ^2}{b}\right]\overline J(PX)=0}.
\end{equation}
Because $PX$ and $\overline J(PX)$ are linearly independent, \eqref{4.18} implies that $k=0$ and \\
$\displaystyle{c+(2n-3)\overline {\widetilde \nu }+\frac{\rho ^2}{b}=0}$. The last equality is equivalent to
$\displaystyle{b=-\frac{\rho ^2}{c+(2n-3)\overline {\widetilde \nu }}}$. Since $\overline {\widetilde \nu }$ is a constant and 
by using \eqref{2.21},  we find 
$\displaystyle{\tau (\xi )=-\frac{1}{\rho }\xi (\rho )}$. Substituting $\tau (\xi )$ in \eqref{3.7} we obtain
$\displaystyle{\overline {\widetilde \nu }=\frac{\rho ^2}{b}}$. The implication $(iii)\Longrightarrow (v)$ follows 
from \eqref{3.10}.
\par
Analogously, by using Theorem \ref{thm 3.3}, we can prove the theorem in the case when $(M,\widetilde g)$ is a totally 
umbilical radical transversal lightlike hypersurface of $\overline M$.
\end{proof}

\section{An example of a totally umbilical radical transversal lightlike hypersurface, which is locally symmetric,
semi-symmetric, Ricci semi-symmetric and almost Einstein}
 
\begin{pro}\label{pro 5.1}
Let $(M,g)$  (resp. $(M,\widetilde g)$) be the hypersurface from Theorem \ref{thm 3.2} (resp. Theorem \ref{thm 3.3}). 
If $b$ and $\rho $ (resp. $\widetilde b$ and $\widetilde \rho $) are constants, then 
$(M,g)$ $(resp.\, (M,\widetilde g))$ is locally symmetric, semi-symmetric, Ricci semi-symmetric and almost Einstein.
\end{pro}
\begin{proof}
By using \eqref{2.21} (resp. \eqref{3.20}), from \eqref{3.7} (resp. \eqref{3.21}) we obtain 
$\displaystyle{\overline {\widetilde \nu }=\frac{\rho ^2}{b}}$ 
$\left(resp.\, \, \displaystyle{\overline \nu=\frac{\widetilde {\rho}^2}{\widetilde b}}\right)$. Then our assertion follows from
Theorem \ref{thm 1.1}.
\end{proof}
 
Further, we construct a K\"ahler-Norden manifold of constant totally real sectional curvatures and its hypersurface, which 
satisfies the conditions of Proposition \ref{pro 5.1}.
\begin{example}\label{exa 5.1}
Consider the Lie group $\overline G$ of all $(2\times 2)$ complex upper triangular matrices with a determinant 1, i.e.
$\overline G=\left\{
\left(\begin{array}{ll}
z_1 & z_2
\cr 0 & z_1^{-1} 
\end{array}\right) : z_1\in \mathbb{C}\backslash\{0\}, z_2\in \mathbb{C}\right\}$. The real Lie algebra 
$\overline {\mathfrak {g}}$ of $\overline G$ 
consists of all $(2\times 2)$ complex upper triangular traceless matrices, i.e.\\
$\overline {\mathfrak {g}}=\left\{
\left(\begin{array}{lr}
w_1 & w_2
\cr 0 & -w_1 
\end{array}\right) : w_1, w_2\in \mathbb{C}\right\}$. The Lie algebra $\overline {\mathfrak {g}}$ is spanned by the left
invariant vector fields $\{X_1,X_2,X_3,X_4\}$, where
$$
X_1=\left(\begin{array}{lr}
i & 0
\cr 0 & -i 
\end{array}\right) , \, \,
X_2=\left(\begin{array}{ll}
0 & i
\cr 0 & 0 
\end{array}\right) , \, \,
X_3=\left(\begin{array}{lr}
1 & 0
\cr 0 & -1 
\end{array}\right) , \, \,
X_4=\left(\begin{array}{ll}
0 & 1
\cr 0 & 0 
\end{array}\right) .
$$
The non-zero Lie brackets of the basic vector fields are given by
\begin{equation}\label{5.1}
\left[X_1,X_2\right]=-[X_3,X_4]=-2X_4\, ;\qquad \left[X_1,X_4\right]=-[X_2,X_3]=2X_2\, .
\end{equation}
Since $\overline {\mathfrak {g}}$ is a complex Lie algebra, we define a complex structure $\overline J$ on 
$\overline {\mathfrak {g}}$ by $\overline J\overline X=-i\overline X$ 
for any left invariant vector field $\overline X$ belonging to $\overline {\mathfrak {g}}$. Hence we have 
$[\overline J\overline X,\overline Y]
=\overline J[\overline X,\overline Y]$ for any $\overline X,\overline Y\in \overline {\mathfrak {g}}$, i.e.
$\overline J$ is a bi-invariant complex structure. Now, we define a left invariant metric $\overline g$ on 
$\overline {\mathfrak {g}}$ by
\begin{equation}\label{5.2}
\begin{array}{ll}
\overline g(X_i,X_i)=-\overline g(X_j,X_j)=1; \quad  i=1,2; \quad j=3,4; \\
\overline g(X_i,X_j)=0; \quad i\neq j; \quad i,j=1,2,3,4.
\end{array}
\end{equation}
From $\overline JX_1=X_3, \, \overline JX_2=X_4$ it follows that the introduced metric is a Norden metric on $\overline G$. Thus
$(\overline G,\overline J,\overline g,\overline {\widetilde g})$ is a 4-dimensional complex manifold with Norden metric. Since
the metric $\overline g$ is left invariant, for the Levi-Civita connection $\overline \nabla $ of $\overline g$ we have
\begin{equation}\label{5.3}
2\overline g(\overline \nabla _{\overline X}\overline Y,\overline Z)=\overline g
([\overline X,\overline Y],\overline Z)+\overline g([\overline Z,\overline X],\overline Y)+
\overline g([\overline Z,\overline Y],\overline X)
\end{equation}
for any $\overline X, \overline Y, \overline Z\in \overline {\mathfrak {g}}$. By using the fact that $\overline J$ is bi-invariant and \eqref{5.3} we get $F(\overline X,\overline Y,\overline Z)=0$. Hence 
$(\overline G,\overline J,\overline g,\overline {\widetilde g})$ is a K\"ahler-Norden manifold. Next, by using \eqref{5.1},
\eqref{5.2} and \eqref{5.3} we obtain the components of $\overline \nabla $. For the non-zero of them we have
\begin{equation}\label{5.4}
\begin{array}{ll}
\overline \nabla _{X_2}X_1=-\overline \nabla _{X_4}X_3=2X_4 \, ;  \quad 
\overline \nabla _{X_2}X_2=-\overline \nabla _{X_4}X_4=-2X_3 \, ; 
\\
\overline \nabla _{X_2}X_3=\overline \nabla _{X_4}X_1=-2X_2 \, ;  \quad 
\overline \nabla _{X_2}X_4=\overline \nabla _{X_4}X_2=2X_1 \, .
\end{array}
\end{equation}
We denote $\overline R_{ijkl}=\overline R(X_i,X_j,X_k,X_l)$, $(i,j,k,l=1,2,3,4)$, where $\overline R$ is the curvature tensor
on $\overline G$. Taking into account \eqref{5.4} we get the non-zero components of $\overline R$
\begin{equation}\label{5.5}
\overline R_{1441}=\overline R_{2332}=\overline R_{1423}=-\overline R_{1221}=-\overline R_{3443}=-\overline R_{1234}=-4 .
\end{equation}
By using \eqref{3.3}, \eqref{5.2} and \eqref{5.5} we find $\overline R(\overline X,\overline Y,\overline Z,\overline W)=4
\left\{(\overline \pi _1-\overline \pi _2)(\overline X,\overline Y,\overline Z,\overline W)\right\}$ for any 
$\overline X,\overline Y,\overline Z,\overline W\in \overline {\mathfrak {g}}$. Then from Theorem \ref{thm 3.1} it follows that 
$\overline G$ is of constant totally real sectional curvatures $\overline \nu =4$ and $\overline {\widetilde \nu }=0$ with 
respect to $\overline g$. Further we consider the Lie subalgebra $\mathfrak {g}$ of $\overline {\mathfrak {g}}$ which is
spanned by $\{X_2,X_3,X_4\}$. The corresponding to $\mathfrak {g}$ Lie subgroup $G$ of $\overline G$ is given by 
$G=\left\{
\left(\begin{array}{ll}
a & z
\cr 0 & a^{-1} 
\end{array}\right) : a\in \mathbb{R}\backslash\{0\}, z\in \mathbb{C}\right\}$. Denote by $g$ and $\widetilde g$ 
the induced metrics on $G$ of
$\overline g$ and $\overline {\widetilde g}$, respectively. Hence $(G,g)$ and $(G,\widetilde g)$ are hypersurfaces of
$\overline G$. By using $\overline JX_1=X_3, \overline JX_2=X_4, \overline JX_3=-X_1, \overline JX_4=-X_2$ and \eqref{5.2} 
we get
\begin{equation}\label{5.6}
\begin{array}{ll}
\widetilde g(X_2,X_2)=\widetilde g(X_3,X_3)=\widetilde g(X_4,X_4)=\widetilde g(X_2,X_3)=\widetilde g(X_3,X_4)=0; \\
\widetilde g(X_2,X_4)=-1
\end{array}
\end{equation}
and
\begin{equation}\label{5.7}
\overline {\widetilde g}(X_1,X_1)=\overline {\widetilde g}(X_1,X_2)=\overline {\widetilde g}(X_1,X_4)=0; \quad
\overline {\widetilde g}(X_1,X_3)=-1.
\end{equation}
The equalities \eqref{5.6} imply that $(G,\widetilde g)$ is a lightlike hypersurface of $\overline G$. Taking into account 
\eqref{5.7} it is easy to check that the normal space ${\mathfrak {g}}^{\widetilde {\bot }}$ of $\mathfrak {g}$ with respect to
$\overline {\widetilde g}$ is spanned by $\{\xi =-X_3\}$ and the screen distribution $S(\mathfrak {g})$ is spanned by
$\{X_2,X_4\}$. The pair $\{N=X_1,\xi \}$ satisfies the conditions \eqref{3.19}. Since
$\overline J\xi =N$, we have $\widetilde b=1$. Hence $(G,\widetilde g)$ is a radical transversal lightlike hypersurface of $\overline G$. As the Levi-Civita connections $\overline \nabla $ and $\overline {\widetilde \nabla }$ on $\overline G$ coincide, from \eqref{5.4}
we obtain
\begin{equation}\label{5.8}
\overline {\widetilde \nabla }_\xi \xi =0, \quad \overline {\widetilde \nabla }_{X_2} \xi =2X_2, \quad
\overline {\widetilde \nabla }_{X_4} \xi =2X_4 .
\end{equation}
Now, \eqref{2.8} and \eqref{5.8} imply that $A^*_\xi X=A^*_\xi PX=-2PX$ for any $X\in \mathfrak {g}$, i.e. $(G,\widetilde g)$
is totally umbilical and $\widetilde \rho =-2$. 
Then, according to Proposition \ref{pro 5.1}, the totally umbilical radical transversal lightlike hypersurface 
$(G,\widetilde g)$ of $\overline G$ is locally symmetric, semi-symmetric, Ricci semi-symmetric and almost Einstein. 
\end{example}

\end{document}